\documentclass{amsart}

\usepackage{amssymb}
\usepackage{amsmath}
\usepackage{enumerate}
\usepackage{array}
\usepackage{todonotes}
\usepackage{color}
\usepackage{dsfont}
\usepackage{comment} 
\usepackage[utf8]{inputenc}
\usepackage{hyperref}
\usepackage{blindtext}
\usepackage[margin=2.5cm]{geometry}

\newtheorem{Theorem}{Theorem}

\newtheorem{Corollary}{Corollary}
\newtheorem{Proposition}{Proposition}
\newtheorem{Remark}{Remark}

\numberwithin{equation}{section}

\newcommand{\Lip}{\textrm{Lip}}

\newcommand{\Prob}{\mathbb{P}}
\newcommand{\EXP}{\mathbb{E}}

\newcommand{\mP}{\mathbb P}
\newcommand{\mE}{\mathbb E}

\newcommand{\fl}[1]{\lfloor {#1} \rfloor}

\newcommand{\pe}{p}

\title[Exponential decay...]{Exponential decay of correlations for random interval diffeomorphisms}
\author{Klaudiusz Czudek}
\address{Klaudiusz Czudek, Institute of Applied Mathematics, Faculty of Physics and Applied Mathematics, Gda{\'n}sk University of Technology, ul. Gabriela Narutowicza 11/12, 80-223 Gda{\'n}sk, Poland}
\email{klaudiusz.czudek@gmail.com}



\begin{document}

\begin{abstract}
We consider a finite number of orientation preserving $C^2$ interval diffeomorphisms and apply them randomly in such a way that the expected Lyapunov exponents at the boundary points are positive. We prove the exponential decay of correlations for Lipschitz observables with respect to the unique stationary measure supported on the interior of the interval. The key step is to show the exponential synchronization in average.
\end{abstract}

\maketitle

\section{Introduction}

This paper is devoted to iterated function systems (\textbf{IFS} for short) consisting of orientation preserving homeomorphisms $f_1,\cdots, f_m$  of the interval $[0,1]$ with a probability vector $(p_1,\cdots,p_m)$ assigned to them. We investigate the ergodic properties of this system when the functions are applied randomly in i.i.d. fashion.  Such systems drew a lot of attention of mathematical community because of their ergodic and dynamical properties (e.g.  \cite{Dolgopyat_Krikorian_2007}, \cite{Ilyashenko_2010}, \cite{Diaz_Gelfert_2012}, \cite{Gharaei_Homburg_2017}, \cite{DeWitt_Dolgopyat_2024}) and their relevance to the study of group representations on manifolds (e.g. \cite{Navas_2011}, \cite{Navas_2018}).
\vspace{0.5cm}

Let $M$ be an arbitrary Polish space, let $f_1,\cdots, f_m : M\rightarrow M$ be continuous transformations on $M$ and let $(p_1,\cdots, p_m)$ be a probability vector. Throughout the paper $\mathbb{N}=\{1, 2, \ldots\}$, $\Sigma=\{1,\cdots, m\}^\mathbb{N}$, $\mathcal{F}$ is the product $\sigma$-field on $\Sigma$, $(\mathcal{F}_n)\subseteq \mathcal{F}$ is the standard filtration. The measurable space $(\Sigma, \mathcal{F})$ is equipped with the infinite product of the measure $(p_1,\cdots,p_m)$ denoted by $\Prob$. The average with respect to that measure is denoted by $\mathbb{E}$.  Let $f^n_\omega(x)=f_{\omega_n}\circ\cdots \circ f_{\omega_1}(x)$ for $n\ge 1$, $\omega=(\omega_1,\omega_2,\cdots)$ and $f^0_\omega(x)=x$. By $\theta$ we denote the left shift on $\Sigma$. Then $(f^n_\omega(x))_{n\ge 0}$ is the Markov process  started at $x\in M$ with transition kernel
\begin{equation}\label{Markov}
p(x,\cdot) = \sum_{i=1}^m p_i \delta_{f_i(x)}(\cdot).
\end{equation}

Let $\mathcal{B}(M)$ denote the space of Borel measurable bounded real functions on $M$, and let $\mathcal{M}$ stands for the space of Borel probability measures on $M$. With the Markov process above we associate the Markov operator $P:\mathcal{M}\rightarrow \mathcal{M}$
\begin{equation}
\label{E:Intro1}
P\mu(A)=\sum_{i=1}^m p_i \mu(f_i^{-1}(A)),
\end{equation}
and its predual operator $U:\mathcal{M}\rightarrow \mathcal{M}$
\begin{equation}
\label{E:Intro2}
U\varphi(x)=\sum_{i=1}^m p_i \varphi(f_i(x)).
\end{equation}
Those operators obey the relation
\begin{equation}
\label{E:Intro3}
\int_M \varphi dP\mu = \int_M U\varphi d\mu, \quad \mu\in \mathcal{M}, \varphi\in \mathcal{B}(M).
\end{equation}
\vspace{0.5cm}

The central notion of the present paper is synchronization. We say the system of homeomorphisms of $M$ with probability vector is synchronizing if $|f^n_\omega(x)- f^n_\omega(y)|\to 0 $ a.s. for any $x,y\in M$. Homburg has shown \cite{Homburg_2018}, Theorem 1.1., that on every compact manifold $M$ there exists a $C^1$-open set of systems of $C^2$ diffeomorphism such that the corresponding (minimal) iterated function system is synchronizing on some open and dense subset of $M$.

This property turns out to be quite common if $M$ is one-dimensional. The first result in that direction has been obtained by  Antonov \cite{Antonov_1984}, whose theorem is as follows. Let $M=\mathbb{S}^1$ and let $f_1,\cdots, f_m$ be circle homeomorphisms. Let us assume that the \textbf{IFS} generated by $f_1,\cdots, f_m$ is minimal (i.e. there is no closed invariant proper subset of $M$) and the \textbf{IFS} generated by $f^{-1}_1, \cdots, f^{-1}_m$ is minimal too. Then exactly one of the following properties holds: the \textbf{IFS} is topologically conjugated to an \textbf{IFS} of rotations, the action is synchronizing or there exists $g\in$ Homeo($M$) of a finite order $p\ge 2$ commuting with  all $f_i$'s. In the third case the system can be factorized by identifying the orbits of $g$  to obtain another system that appears synchronizing. Malicet \cite{Malicet_2017} improved Antonov's result by showing that in the second and third case of this trichotomy the rate of synchronization is exponential (without any additional smoothness assumptions on the homeomorphisms), i.e. he showed that there exists $q\in (0,1)$ such that for every $x\in \mathbb{S}^1$ and $\Prob$-a.e. $\omega\in\Sigma$ there exists a neighbourhood $I$ of $x$ such that $|f^n_\omega(I)|\le q^n$ for every $n\ge 1$ (an alternative proof of Malicet's result has been given in recent paper \cite{Choi_2025}).  That condition has been used to show (under additional minimality assumption) the uniqueness of stationary measures for the corresponding Markov chain \eqref{Markov} (Corollary 2.3 \cite{Malicet_2017}), the central limit theorem and the law of the iterated logarithm \cite{Szarek_Zdunik_2017}, \cite{Szarek_Zdunik_2020} for additive functionals of Markov chains \eqref{Markov}. Later Gelfert and Salcedo \cite{Gelfert_Salcedo_2024} established limit laws in the case of arbitrary proximal\footnote{A system is proximal if for every $x,y\in M$ there exist $\omega\in\Sigma$ and $n_k\to \infty$ such that $d(f^{n_k}_\omega(x), f^{n_k}_\omega(y)) \to 0$.} systems with so called local contraction property (see \cite{Gelfert_Salcedo_2024} for the definition). In \cite{Gelfert_Salcedo_2023} the authors discussed various notions of contractivity of systems on $\mathbb{S}^1$ and relations between them and showed that proximal systems of $C^1$ circle diffeomorphisms that do not have a common invariant measure satisfy 
\begin{equation}
\label{Expon}
\EXP d(f^n_\omega(x), f^n_\omega(y))^\beta \le Cd(x,y)^\beta q^n
\end{equation}
for some $q\in (0,1)$, $\beta\in (0,1)$, $C\ge 1$ and all pairs $x,y\in \mathbb{S}^1$, where $d$ stands for the Euclidean metric on $\mathbb{S}^1$. The same result is proven (although in slightly disguised form) in Proposition 4.18 \cite{Gorodetski_Kleptsyn_2020}. Le Page \cite{LePage_1982} proved this result much earlier under additional assumption that $f_i$'s are elements of PSL($2,\mathbb{R}$). It is not known it Malicet's result imply \eqref{Expon} without any smoothness assumptions -- it would imply the affirmative answer to Question 1 in \cite{Barrientos_Malicet_2024} at least in the case of systems on the circle.
\vspace{0.5cm}

When $M=[0,1]$ there are obviously two fixed points $0,1$, so that the interesting dynamics takes place in the interior of $M$. A stationary measure $\mu$ (i.e. a fixed point of $P$) with $\mu((0,1))=1$ does not necessarily exists. However, if we assume that the expected Lyapunov exponents at the boundary
\begin{equation}
\label{E:exp_lyapunov}
\Lambda_0=\sum_{i=1}^m p_i f_i'(0), \quad \Lambda_1=\sum_{i=1}^m p_i f_i'(1)
\end{equation}
are positive, then there exists such a measure $\mu_\ast$ \cite{Gharaei_Homburg_2017}. There is an easy argument that $\mu_\ast$ is necessarily unique then (Theorem 1 \cite{Czudek_Szarek_2020}). The existence of $\mu_\ast$ implies Malicet's exponential contraction (Corollary 2.13 in \cite{Malicet_2017}) for the action of homeomorphisms. This result has been exploited to show the central limit theorem \cite{Czudek_Szarek_2020} and the law of the iterated logarithm \cite{Czudek_Wojewodka_Szarek_2020}. If $f_i$'s are $C^1$ diffeomorphisms it has been proven independently by Gharaei, Homburg (see Theorem 4.1 in \cite{Gharaei_Homburg_2017}) that $|f^n_\omega(x)- f^n_\omega(y)|$ decays exponentially fast a.s. for any two $x,y\in (0,1)$. Their argument is based on the Baxendale theorem \cite{Baxendale_1989} (see also Lemma 4.1 in \cite{Gharaei_Homburg_2017}) saying that the volume Lyapunov exponent
\begin{equation}
\label{E:Intro4}
\Lambda = \sum_{i=1}^m p_i \int_0^1 \log f_i'(x) \mu_\ast(dx)
\end{equation}
is negative. The main result of the present paper says that the relation similar to \eqref{Expon} holds on the open interval $(0,1)$. The difficulty comes from the lack of compactness. 
\vspace{0.5cm}

Our standing assumptions are
\begin{enumerate}
\item[(A1)] $f_1,\cdots, f_m\in \textrm{Diff}^2_+([0,1])$ -- the space of $C^2$ increasing interval diffeomorphisms,
\item[(A2)] for every $x\in (0,1)$ there exist $i, j$ with $f_i(x)<x<f_j(x)$,
\item[(A3)] the expected Lyapunov exponents $\Lambda_0$, $\Lambda_1$ in \eqref{E:exp_lyapunov} are positive.
\end{enumerate}

\begin{Theorem}
\label{T:main}
Let $f_1,\cdots, f_m$,  $p_1\cdots, p_m$ satisfy (A1)-(A3). For every $x, y \in (0,1)$ there exist $C\ge 1$, $q\in (0,1)$ such that
\begin{equation}
\label{E:theorem_main}
\EXP|f^n_\omega(x) - f^n_\omega(y)| \le Cq^n \quad \textrm{for $n\ge 1$.}
\end{equation}
\end{Theorem}

To show Theorem \ref{T:main} we follow the strategy from Proposition 4.18 \cite{Gorodetski_Kleptsyn_2020} and from \cite{LePage_1982} rather than \cite{Gelfert_Salcedo_2023}. Roughly speaking, the idea from \cite{Gorodetski_Kleptsyn_2020} and \cite{LePage_1982} can be repeated on an arbitrarily large compact subinterval $[a,1-a]\subseteq (0,1)$, i.e. we can show that there exist $n_0$, $q_1\in (0,1)$, and $\beta\in (0,1)$ such that $\EXP | f^{n_0}_\omega(x)- f^{n_0}_\omega(y)|^\beta \le q_1|x-y|^\beta$ for every $x,y\in [a,1-a]$. The issue is that $[a,1-a]$ in not invariant for the semigroup action and some trajectories leave it after $n_0$ iterations. In order to deal with the problem we show Proposition \ref{P:expansion_return}, which says that the distance between points is not too much expanded at the moment of the first common return to $[a,1-a]$.

Theorem \ref{T:main} can be used to show the following result about the decay of correlations for the averaging operator on the space of Lipschitz observables.

\begin{Corollary}
\label{C:2}
Let $f_1,\cdots, f_m$,  $p_1\cdots, p_m$ satisfy (A1)-(A3). For every $\pe \ge 1$ there exist $q\in (0,1)$, $C\ge 1$ such that 
$$\bigg|\int_0^1  U^n \varphi(x) \psi(x) d\mu_\ast(x) - \int_0^1 \varphi(x) \textrm{d}\mu_\ast(x) \int_0^1 \psi(x) \textrm{d}\mu_\ast(x)  \bigg| \le C \|\varphi\|_{Lip} \|\psi\|_{Lip} q^{n} \quad \textrm{for $n\ge 1$.}$$
for every Lipschitz functions $\varphi$, $\psi$, where $\|\varphi\|_{\textrm{Lip}}=\|\varphi\|_\infty+\textrm{Lip}(\varphi)$ and $\textrm{Lip}(\varphi)$ is the optimal Lipschitz constant of $\varphi$.
\end{Corollary}

\begin{Remark}
The functions in Corollary \ref{C:2} are assumed to be Lipschitz with respect to the standard Euclidean metric on $[0,1]$, i.e. $d(x,y)=|x-y|$.
\end{Remark}

\begin{Remark}
Note that the constants $C$, $q$ appearing in Corollary 1 are in general different from the one in Theorem \ref{T:main}.
\end{Remark}

\section{Auxiliary results}

\subsection{Basic facts}
Let us recall once again the facts that are going to be used in the sequel. Under assumptions (A1)-(A3) there exists a fixed point $\mu_\ast$ of the operator $P$ in \eqref{E:Intro1} with $\mu_\ast((0,1))=1$. Among all measures $\mu\in\mathcal{M}$ with $\mu((0,1))=1$ measure $\mu_\ast$ is unique. Moreover, under (A1)-(A3) the volume Lyapunov exponent \eqref{E:Intro4} is negative (Lemma 4.1 in \cite{Gharaei_Homburg_2017}) and $|f^n_\omega(x)-f^n_\omega(y)| \to 0$ a.s. for any two $x,y\in (0,1)$ (Theorem 4.1 in \cite{Gharaei_Homburg_2017}). In the sequel we refer to the last property as the synchronization.

Let
\begin{equation}
\label{E:pmalpha}
\mathcal{P}_{M,\alpha}=\big\{\mu\in\mathcal M : \forall_{x\in (0,1)} \mu\big((0,x]\big)\le Mx^\alpha \ \textrm{and} \ \mu\big([1-x,1)\big)\le Mx^\alpha\ \big\}.
\end{equation}
The following result can easily inferred from Lemma 1 in \cite{Czudek_Szarek_2020}. We give the proof here for the convenience of the reader.

\begin{Proposition}
\label{P:basic1}
Let $f_1,\cdots, f_m$,  $p_1\cdots, p_m$ satisfy (A1)-(A3). There exists $\alpha\in(0,1)$ and $c\in (0,1)$ such that for every $a>0$ sufficiently small
\begin{equation}\label{E:fast_escape}
\Prob\bigg( \bigcap_{j=1}^n \{ f^j_\omega(x)<a \} \bigg) \le \bigg(\frac{a}{x}\bigg)^\alpha c^n, \quad
\Prob\bigg( \bigcap_{j=1}^n \{ f^j_\omega(1-x)>1-a \} \bigg) \le \bigg(\frac{a}{x}\bigg)^\alpha c^n \quad \textrm{for $x\in(0,1)$}.
\end{equation}
Moreover, for every $a>0$ sufficiently small there exists $M\ge 1$ such that the class $\mathcal P_{M,\alpha}$ is invariant under the action of the corresponding Markov operator $P$ and every measure supported on $[a,1-a]$ belongs to this class. Consequently, the unique stationary measure $\mu_\ast$ belongs to $\mathcal{P}_{M, \alpha}$.
\end{Proposition}
\begin{proof}
Let us consider the system of inverse functions $f_1^{-1},\ldots, f_m^{-1}$ with the same probability vector. This new system has negative expected Lyapunov exponents at $0$ and $1$ by (A1)-(A3), therefore we can find numbers $a_1, \ldots, a_m$, $b_1, \ldots, b_m$ and $\zeta>0$ with
\begin{enumerate}
\item $\sum_{i=1}^m p_i \log a_i<0$ and $\sum_{i=1}^m p_i \log b_i<0$
\item $f_i(x)\ge a_i x$ and $f_i(1-x)\le 1-b_i x$ for $i=1,\ldots, m$, $x\le\zeta$.
\item $f_i^{-1}(x)\le a_i^{-1} x$ and $f_i^{-1}(1-x)\ge 1-b_i^{-1} x$ for $i=1,\ldots, m$, $x\le\zeta$,
\end{enumerate}
Using the Taylor formula applied to the function $\alpha \longmapsto a^{-\alpha}$ we find $\alpha>0$ and $c<1$ with
\begin{equation}
\label{E:constant_c}
\EXP a_\omega^{-\alpha} = \sum_{i=1}^m p_i a_i^{-\alpha}<c \ \textrm{and} \ \EXP b_\omega^{-\alpha} = \sum_{i=1}^m p_i b_i^{-\alpha}<c,
\end{equation}
where $a_\omega := a_{\omega_0}$, $b_\omega :=b_{\omega_0}$. Let $a$ be an arbitrary positive number less than $\zeta$. We are going to show the first inequality in \eqref{E:fast_escape} (the proof of the second one is obtained by an easy modification). Let 
$$B_n=\{f_\omega^1(x)<a\}\cap \cdots \cap \{f_\omega^n(x)<a\}$$ 
and $a_\omega^n = a_{\omega_0}\cdots a_{\omega_{n-1}}$, $n\ge 1$. Clearly $B_n\subseteq \{a_\omega^n x < a\} = \{ \big(a_\omega^n \big)^{-\alpha}>(a/x)^{-\alpha} \}$ by (2), thus by the Chebyshev inequality, \eqref{E:constant_c} and induction argument
$$\Prob(B_n) \le \bigg(\frac{a}{x}\bigg)^\alpha\EXP \big(a_\omega^n\big)^{-\alpha}<\bigg(\frac{a}{x}\bigg)^\alpha c^n.$$
Take $M$ such that $Ma^\alpha=1$, and take arbitrary $x\in (0,1)$. If $x\ge a$, then $Mx^\alpha\ge Ma^\alpha=1$, hence $P\mu((0,x])\le Mx^\alpha$ for every $\mu\in \mathcal M$. If $x<a$ and $\mu\in \mathcal P_{M,\alpha}$, then
$$P\mu((0,x])=\sum_{i=1}^m p_i\mu((0,f^{-1}_i(x)])\le \sum_{i=1}^m p_i\mu((0,a_i^{-1}x]) $$
$$\le \sum_{i=1}^m p_i M a^{-\alpha}_ix^\alpha<Mx^\alpha,$$
by the choice of $\alpha$. The analogous computation for $P\mu([1-x,1))$ proves the $P$-invariance of $\mathcal P_{M,\alpha}$.

Now fix $a\in (0,1/2)$ and $M$ given above. Take any $\mu$ supported on $[a,1-a]$ and consider the sequence
$$\mu_n = \frac{1}{n}(\mu+P\mu+\cdots+P^{n-1}\mu).$$
The class $\mathcal{P}_{M,\alpha}$ is clearly convex and weakly-$\ast$ closed, therefore each accumulation point of $(\mu_n)$ is in $\mathcal{P}_{M, \alpha}$. The set of accumulation points is nonempty by the compactness of $\mathcal{P}_{M, \alpha}$ and similarly as in the Krylov-Bogoliubov method every accumulation point of $(\mu_n)$ is a stationary measure. Thus by uniqueness of the stationary measure $\mu_\ast \in \mathcal{P}_{M,\alpha}$.
\end{proof}

We shall need the following result, which is the consequence of the uniqueness of the stationary distribution. The proof is analogous to the proof of Proposition 4.1.13 in \cite{Hasselblatt_Katok_1995} and is omitted.

\begin{Proposition}\label{P:uniform}
Let $f_1, \ldots, f_m$, $p_1,\ldots, p_m$ satisfy (A1)-(A3). Let $\varphi:[0,1]\rightarrow \mathbb{R}$ be a continuous function. Then
$$
\frac{\EXP \varphi(x)+\cdots+\EXP\varphi(f^{n-1}_\omega(x))}{n}
$$
converge uniformly to $\int \varphi d\mu_\ast$ on every compact set $K\subseteq (0,1)$.
\end{Proposition}

\subsection{Large deviations}

\begin{Proposition}
\label{P:large_deviation}
 Let $f_1, \ldots, f_m$, $p_1,\ldots, p_m$ satisfy (A1)-(A3). Then there exist $C$ and $q$ such that for every $a>0$ sufficiently small and for every $x \in [a,1-a]$
 $$\Prob\bigg(\frac{\#\{i\le n : f^n_\omega(x)\in[a,1-a]\}}{n}\le 3/4\bigg)\le Cq^n$$
 for every $n$.
\end{Proposition}
\begin{proof}
Let $c\in (0,1)$ and $\alpha\in (0,1)$ be the constants given in Proposition \ref{P:basic1}, and let $a$ be so small that Proposition \ref{P:basic1} applies. For $x\in [a,1-a]$ and $\omega\in\Sigma$ define
$$t_0=0,$$
  $$s_1(\omega, x)=\min\{k > t_0 : f^k_\omega(x)\not\in [a,1-a]\},$$
  $$t_n(\omega, x)=\min \{ k>s_n : f^k_\omega(x)\in [a,1-a] \}, \quad n\ge 1,$$
  $$s_n(\omega, x)=\min \{ k>t_{n-1} : f^k_\omega(x)\not\in [a,1-a] \}, \quad n\ge 2.$$
Let $\tau_n=s_n-t_{n-1}$ and $\sigma_n=t_n - s_n$, $n\ge 1$. Let $\rho_n(\omega)=\max\{k : s_k\le n\}$.
We start with the easy inequality 
\begin{equation}
\label{E:large_deviation1}
\frac{\#\{i\le  n : f_\omega^i(x) \not\in [a,1-a]\}}{n}\le \frac{\sigma_1+\cdots+\sigma_{\rho_n}}{\tau_1+\sigma_1+\cdots+\tau_{\rho_n}+\sigma_{\rho_n}}=\frac{\#\{i\le t_{\rho_n} : f_\omega^i(x)\not\in [a,1-a]\}}{t_{\rho_n}}, \quad \omega\in\Sigma.
\end{equation}
Indeed, if $n\in [t_{\rho_n}, s_{\rho_n+1})$, then
$\#\{i\le  n : f_\omega^i(x) \not\in [a,1-a]\} = \sigma_1+\cdots+\sigma_{\rho_n}$
and $\tau_1+\sigma_1+\cdots+\tau_{\rho_n}+\sigma_{\rho_n} \le n$. If $n\in [s_{\rho_n}, t_{\rho_n})$, then \eqref{E:large_deviation1} follows from the inequality
$a/b \le (a+c)/(b+c)$ (valid for $0\le a< b$, $c>0$).

Put
 $$h=\min\bigg\{\min_{i=1,\ldots, m} \frac{f_i'(0)}{2} , \min_{i=1,\ldots, m} \frac{f_i'(1)}{2} \bigg\}.$$
If $a$ is sufficiently small then by Proposition \ref{P:basic1} and the strong Markov property we can write for each $k\ge 1$
 $$
 \mP(\sigma_k>n)  
 \le\frac{a^\alpha}{(ha)^\alpha}c^n
 =h^{-\alpha}c^n.
 $$
 Note that this inequality holds as long as $a$ is sufficiently small and the right-hand side does not depend on the choice of $a$. In view of these remarks, if $\gamma_1$ is such that $e^{\gamma_1}c<1$, then by the strong Markov property
 $$\mE\big( e^{\gamma_1\sigma_k} \ | \ \mathcal{F}_{s_k}\big) \le \sum_{n=1}^\infty e^{\gamma_1 n} \mP (\sigma_k \ge n) \le \sum_{n=1}^\infty e^{\gamma_1 n} c^{n-1}   \le e^{L_1} \quad \textrm{a.s.}$$
 for some $L_1$ and every $k\ge 1$. By conditioning argument 
 $\mE e^{\gamma_1(\sigma_1+\cdots+\sigma_k)} \le e^{L_1k}$ for $k\ge 1$, and thus by Chebyshev's inequality for $L_2 :=\frac{2 L_1}{\gamma_1}$ and $q_2=e^{- L_1}$ we have
 \begin{equation}\label{E:large_deviation2}
 \mP\bigg(\sigma_1+\cdots+\sigma_k> k L_2 \bigg)\le e^{-kL_1}=q_2^k \quad k\ge 1.
 \end{equation}
Let us stress once again, the constants appearing in \eqref{E:large_deviation2} do not depend on $a$ as long as $a$ is sufficiently small.

Take $b>0$, $a_1, \ldots, a_m$, $b_1,\ldots, b_m$ such that
 \begin{enumerate}[(i)]
  \item $f_i(x)\ge a_ix$ and $f_i(1-x)\le 1-b_i x$ for $x<b$, and
 \item $\sum_{i=1}^m p_i\log a_i>0$ and $\sum_{i=1}^m p_i\log b_i>0$.
 \end{enumerate}
Let us consider $(X_n)_{n\ge 1}$ be i.i.d. so that $X_1$ equals $\log a_i$ with probability $p_i$, $i=1,\cdots, m$. Then $S_n:=X_1+\cdots +X_n \to +\infty$ a.s. as $n\to\infty$ by (ii) and the strong law of large numbers. Thus $(S_n)$ is transient, in particular $\mP(\bigcup_{i=1}^\infty \{S_i< 0\})<1$ (Theorem 8.1 in \cite{Kallenberg_2002}). In other words, there exists $\eta>0$ such that for every $A>0$ the probability that $(S_n)$ enters $[A, +\infty)$ before returning to $0$ is greater than $\eta$.

For $A>0$, $r>0$ let $Y^-_{A, r}=0$ on $\{ \textrm{$(S_n)$ visits $(-\infty, 0)$ before $[A,+\infty)$ } \}$ and $Y_{A, r} =r\wedge \min \{n\ge 1: S_n\ge A\}$ otherwise. By the  condition proved in the preceding paragraph if $A>0$, $r>0$ are large enough then $\mE Y^-_{A, r}> 40 L_2$. In the symmetric way we can define $Y^+_{A, r}$ when $X_1$ takes value $\log b_i$ with probability $p_i$, $i=1,\cdots m$. Let $A$ and $r$ be so large that both $\mE Y^-_{A, r}> 40 L_2$ and $\mE Y^+_{A, r}> 40 L_2$.

Now let us fix $a<b$ both sufficiently small and such that $e^Aa<b$. Define $g_i(x)=a_ix$, $i=1,\ldots, m$. For each $k\ge 1$ define $Y_k^-:=0$ if $k-1$-th return to $[a,1-a]$ is from $(1-a, a)$ and $(g^n_{\theta^{t_{k-1}}\omega}(a))$ visits $(0,a)$ before $(ae^A,1)$ and $Y_k^-:=r\wedge \min \{n\ge 1: g^n_{\theta^{t_{k-1}}\omega}(a)>ae^A\}$ otherwise. Observe that the distributions of $Y_k^-$ and $Y_{A, r}^-$ are the same, $k\ge 1$. 
Moreover $(Y_k^-)$ are independent, bounded by $r$ and satisfy $Y_k\le \tau_k$ provided $f^{t_{k-1}}(x)<a$. By the Cram\'er-Chernoff theorem (Theorem 27.3 in \cite{Kallenberg_2002})
\begin{equation}
\label{E:large_deviation3}
\mP(Y_1^-+\cdots+Y_k^-< 16 k L_2)\le L_3 q_3^k, \quad k\ge 1,
\end{equation}
for some $L_3 \ge 1$, $q_3\in (0,1)$. We assume those constant to satisfy the same estimate for $(Y_k^+)$ (defined in a symmetric fashion in a neighbourhood of 1).

Let $H^-$ denote the event that at least $k/4$ numbers among $t_1, \cdots, t_k$ have the property that $f^{t_j-1}_\omega(x)$ belongs to $(0,a)$. Let $H^+$ denote the event that at least $k/4$ numbers among $t_1, \cdots, t_k$ have the property that $f^{t_j-1}_\omega(x)$ belongs to $(1-a,1)$. Then $H^-\cup H^+=\Sigma$. Clearly
$$Y_{1}^-+\cdots+Y_{i_{\lceil k/4 \rceil}}^-\le \tau_1+\cdots+\tau_k \quad \textrm{a.s. on $H^-$}$$
and
$$Y_{1}^++\cdots+Y_{i_{\lceil k/4 \rceil}}^+\le \tau_1+\cdots+\tau_k \quad \textrm{a.s. on $H^+$},$$
hence from \eqref{E:large_deviation3}
\begin{equation}
\label{E:large_deviation4}
\mP\bigg(\tau_1+\cdots+\tau_k< 4k L_2 \bigg)\le 2 L_3 q_3^{k/4}
\end{equation}
for every positive integer $k$.

Using \eqref{E:large_deviation1} one readily checks that if 
$
\omega \in \{\sigma_1 + \cdots + \sigma_{\rho_n} \le \rho_n L_2 \}
\cap \{ \tau_1+\cdots +\tau_{\rho_n} \ge 4\rho_n L_2 \}
$
then 
$$
\frac{\#\{i\le  n : f_\omega^i(x) \not\in [a,1-a]\}}{n}\le \frac{1}{4}, \quad n\ge 1.
$$
We are left with estimating that the probability of the event
$
C =  \{\sigma_1 + \cdots + \sigma_{\rho_n} > \rho_n L_2 \}
\cup \{ \tau_1+\cdots +\tau_{\rho_n} < 4\rho_n L_2 \}
$
decays exponentially fast. Fix $\lambda\in(0,1)$. Using the Chebyshev inequality similarly as in \eqref{E:large_deviation2} we obtain
\begin{equation}
\label{E:large_deviation6}
\Prob( \rho_n \le \lambda n) \le \Prob( \sigma_1+\cdots+\sigma_{\lambda n} > n )\le e^{ L_1 \lambda n} e^{-\gamma_1 n},
\end{equation}
thus if $L_1\lambda-\gamma_1<0$ (i.e. $\lambda$ was choosen sufficiently small), then 
\begin{equation}
\label{E:large_deviation5}
\Prob(\{\rho_n < \lambda n\}\cap C) \le \Prob(\{\rho_n < \lambda n)\end{equation}
decays exponentially fast. On the other hand for that choice of $\lambda$ we have by \eqref{E:large_deviation2}
$$
\Prob(\{\sigma_1 + \cdots + \sigma_{\rho_n} > \rho_n L_2\} \cap \{\rho_n \ge \lambda n \})
$$
$$
\le \sum_{k=\lambda n}^\infty \Prob(\{\sigma_1 + \cdots + \sigma_k> kL_2 \}\cap \{\rho_n = k\}) \le \sum_{k=\lambda n}^\infty 2L_3q_3^{k/4} \le \frac{2L_3q_3^{\lambda n +1}}{1-q_3^{1/4}},
$$
which decays exponentially fast. A similar reasoning applied to $\tau$ (using \eqref{E:large_deviation4} this time) combined with \eqref{E:large_deviation5} leads to the exponential decay of $\Prob(C)$.
\end{proof}

\subsection{Coupling time}

Given $a>0$ and $x,y \in (0,1)$ with $x<y$ we define $T^a_{x, y}(\omega)$ to be the minimum number $k$ with $a\le f_\omega^k(x)<f_\omega^{k}(y)\le 1-a$.

 \begin{Proposition}
 \label{P:coupling}
  Let $f_1, \ldots, f_m$, $p_1$,$\ldots$, $p_m$ satisfy (A1)-(A3), let $\alpha$ be given in Proposition \ref{P:basic1}. There exist constants $\kappa$ and $C_1>0$ such that if $a>0$ is sufficiently small then for every $x<y$ we have
 $$\mE e^{\kappa T^a_{x,y}}\le C_1 \max \bigg\{\big(a/z\big)^\alpha, 1\bigg\},$$
 where $z=\min\{x, 1-y\}$.
  \end{Proposition}
\begin{proof}
Let us take $a>0$ sufficiently small to satisfy Proposition \ref{P:basic1} and \ref{P:large_deviation}. Let $C$ be the constant given in Proposition \ref{P:large_deviation}. We insist also that the transition from $(0,a)$ to $(1-a,1)$ and from $(1-a, 1)$ to $(0,a)$ is impossible in one step.
By Proposition \ref{P:basic1}
\begin{equation}\label{E:coupling1}
\mP\bigg(\bigcap_{k=0}^{\lfloor n/8 \rfloor} \{f_\omega^k(x)<a\}\cup \bigcap_{k=0}^{\lfloor n/8 \rfloor} \{f_\omega^k(x)>1-a\}\bigg)\le 2a^{\alpha}/z^\alpha c^{\fl{n/8}}, \quad n\ge 1.
\end{equation}
Since transition from $(0,a)$ to $(1-a, 1)$ and from $(1-a, 1)$ to $(0,a)$ is not possible in one step we have
$$
\bigcup_{k=0}^{\lfloor n/8 \rfloor} \{f_\omega^k(x) \ge a\}\cap \bigcup_{k=0}^{\lfloor n/8 \rfloor} \{f_\omega^k(y)\le 1-a\}
=
\bigcup_{k=0}^{\fl{n/8}} \{f_\omega^k(x)\in [a, 1-a]\}\cap \bigcup_{k=0}^{\fl{n/8}} \{f_\omega^k(y)\in [a, 1-a]\}=:R.
$$
Let $A_k(z)\subseteq \Sigma$ be the set of $\omega$'s such that the moment of the first visit of $(f_\omega^i(z))$ in $[a,1-a]$ equals $k$. Let $B_k(z)$ be the set of $\omega$'s such that
$$
\frac{ \{ i\le n-k : f^{k+i}_\omega(x)\not\in [a,1-a] \} }{n} > \frac{1}{4}.
$$
By the Markov property and Proposition \ref{P:large_deviation} we arrive at the inequality
$$
\Prob(A_k(z)\cap B_k(z)) = \EXP \mathds{1}_{A_k(z)} \Prob( B_k(z) | \mathcal{F}_k \} \le \Prob(A_k(z)) Cq^{n-k}, \quad z\in (0,1),
$$
thus
\begin{equation}\label{E:coupling2}
\Prob\bigg( \bigcup_{k=0}^{\lfloor n/8 \rfloor} A_k(z)\cap B_k(z) \bigg) \le Cq^{7/8n}, \quad z\in(0,1).
\end{equation}
Obviously $T^a_{x,y}(\omega)\le n$ for $\omega \in (A_j(x)\setminus B_j(x))\cap (A_i(y)\setminus B_i(y))$ for any $i, j=0,1,\cdots, \lfloor n/8 \rfloor$. Using now \eqref{E:coupling1} and \eqref{E:coupling2} we get
$$\Prob(T^a_{x,y} \ge n) \le  2c^{\lfloor n/8 \rfloor} \bigg( \frac{a}{z} \bigg)^\alpha + Cq^{7/8n} \le C \max\bigg\{ \bigg( \frac{a}{z} \bigg)^\alpha, 1\bigg\} q_4^n, \quad n\ge 1,
$$
for some $q_4\in (0,1)$. This easily implies that there exists $\kappa \in(0,1)$ such that the assertion is satisfied with $C_1=C\frac{2q_4}{1-q_4}$.
\end{proof}

\subsection{Expansion after the first visit}

Given $a, x, y$ let $T^a_{x,y}(\omega)$ be as in the preceding subsection. Let $h>0$ be so small that
$$
h \le \min\bigg\{\min_{i=1,\ldots, m} \frac{f_i'(0)}{2} , \min_{i=1,\ldots, m} \frac{f_i'(1)}{2} \bigg\} \quad \textrm{and} \quad h^{-1} \ge \max\bigg\{\max_{i=1,\ldots, m} 2f_i'(0) , \max_{i=1,\ldots, m} 2f_i'(1) \bigg\}.
$$
Let 
$$
L'=\max_{i=1,\cdots, m}\max_{\xi \in [0,1]} f'(\xi), 
\quad 
L''=\max_{i=1,\cdots, m}\max_{\xi \in [0,1]} |\big(\log f'(\xi)\big)'|.
$$

\begin{Proposition}
\label{P:expansion_return}
Let $f_1,\cdots, f_m$, $p_1,\cdots, p_m$ satisfy (A1)-(A3). Let $C_1, \kappa$ be given in Proposition \ref{P:coupling}. Let $a>0$ be such that Propositions \ref{P:basic1} and \ref{P:coupling} are satisfied and $aL''h^{-1}<\kappa$. Let $u\in (0,a)$, and let $\tilde{\eta}$ be sufficiently small. For every $\beta$ sufficiently small there exists $\delta>0$ such that if $x,y\in [u,1-u]$ with $|x-y|<\delta$ then
\begin{equation}
\label{E:expansion_return1}
\EXP | f^{T^a_{x,y}\wedge n}_\omega(x) - f^{T^a_{x,y}\wedge n}_\omega(y)| ^\beta < (1+\tilde{\eta}) \bigg(\frac{C_1}{h}\bigg)^{3\beta N}|x-y|^\beta,
\end{equation}
for every $n\ge 0$, where $N=-\frac{\log (a/x)}{\log h}$.
\end{Proposition}
\begin{proof}
Let $\alpha, C_1, \kappa$ be given in Proposition \ref{P:basic1} and \ref{P:coupling} respectively.
Let us fix $a>0$ as above, $u\in (0,a)$, $\tilde{\eta}>0$. Let $\beta$ be so small that $e^{-\kappa}(L')^\beta<1$ and let $K_0$ be so large that
\begin{equation}
\label{E:expansion_return2}
\sum_{n=K_0+1}^\infty C_1\bigg(\frac{a}{u}\bigg)^\alpha  e^{-\kappa n}(L')^{\beta n} < \tilde{\eta}.
\end{equation}
The statement is obvious when $x,y \in [a,1-a]$ as the right-hand side of \eqref{E:expansion_return1} is greater than 1, so we can consider only the situation when at least one of $x, y$ is outside $[a,1-a]$. Let us assume therefore that $x<y$ and $x<a$. To simplify the notation let us write $T(\omega)=T^a_{x,y}(\omega)$. Let $\delta$ be so small that $f^n_\omega(y)<ah^{-1}$ and $|f^n_\omega(x)-f^n_\omega(y)|\le \frac{1}{2}$ for $n\le K_0$ and $x,y$ such that $|x-y|<\delta$.

We have clearly
$$\mE \frac{|f^T_{ \omega}(x)-f^T_\omega(y)|^\beta}{|x-y|^\beta}
=\mE\mathds{1}_{\{T\le K_0\}} \frac{|f^T_{ \omega}(x)-f^T_{ \omega}(y)|^\beta}{|x-y|^\beta}
+\mE\mathds{1}_{\{T> K_0\}}  \frac{|f^T_{ \omega}(x)-f^T_{ \omega}(y)|^\beta}{|x-y|^\beta}.$$

 To estimate the first integral let us take $\omega\in\{T\le K_0\}$ and write by the mean value theorem
$$|f^T_{ \omega}(x)-f^T_{ \omega}(y)|=(f^T_\omega)'(\zeta_\omega)|x-y|$$
 and
$$f^T_\omega(x)=f^T_\omega(y) - f^T_\omega(0)=(f^T_\omega)'(\xi_\omega)\cdot x$$
 for some $\zeta_\omega \in [x,y]$ and $\xi_\omega\in [0,x]$. By the choice of $\delta>0$ and the fact that $\omega\in \{T\le K_0\}$, $T$ is the moment of the first visit of $f^n_\omega(x)$ in $[a,1-a]$, thus the value of $f^T_\omega(x)$ cannot be greater than $h^{-1}a$. Therefore $f^T_\omega(x)/x\le h^{-1}a/x$ and
$$ \frac{|f^T_{ \omega}(x)-f^T_{ \omega}(y)|^\beta}{|x-y|^\beta}
=\big((f^T_\omega)'(\zeta_\omega)\big)^\beta
= \bigg(\frac{(f^T_\omega)'(\zeta_\omega)}{(f^T_\omega)'(\xi_\omega)}\bigg)^\beta \bigg(\frac{f^T_\omega(x)}{x} \bigg)^\beta
$$
$$
\le \bigg(\frac{h^{-1}a}{x}\bigg)^\beta\cdot \exp\bigg(\beta \big(\log(f^T_\omega)'(\zeta_\omega)-\log(f^T_\omega)'(\xi_\omega) \big)\bigg).$$
Using $T\le K_0$ and the choice of $\delta$ again we have $|f^n_\omega(\xi_\omega)-f^n_\omega(\zeta_\omega)|\le f^n_\omega(y)\le h^{-1}a$ for $n\le T$. By the chain rule and the definition of $L''$ we obtain
$$| \log(f^T_\omega)'(\zeta_\omega)-\log(f^T_\omega)'(\xi_\omega)|\le TL''h^{-1}a,$$
hence by the assumptions that we made on $a$
$$
\frac{|f^T_{ \omega}(x)-f^T_{ \omega}(y)|^\beta}{|x-y|^\beta}\le \bigg(\frac{h^{-1}a}{x}\bigg)^\beta\cdot \exp(TL''h^{-1} a\beta)
\le 
\bigg(\frac{h^{-1}a}{x}\bigg)^\beta  \exp(T \beta\kappa) ,
$$
for $\omega$ such that $T\le K_0$, hence taking the expectation, applying the Jensen inequality and Proposition \ref{P:coupling} give
\begin{equation}
\label{E:expansion_return3}
\mE \mathds{1}_{\{T\le K_0\}} \frac{|f^T_{ \omega}(x)-f^T_{ \omega}(y)|^\beta}{|x-y|^\beta} \le \bigg(\frac{h^{-1}a}{x}\bigg)^\beta \bigg( \EXP \exp(T\kappa) \bigg)^\beta
\le
 \bigg(C_1 \frac{h^{-1}a^{1+\alpha}}{x^{1+\alpha}}  \bigg)^\beta .
\end{equation}

If $T>K_0$, then
$$\frac{|f^T_{ \omega}(x)-f^T_{ \omega}(y)|^\beta}{|x-y|^\beta}\le (L')^{\beta T}$$
by the definition of $L'$, therefore
\begin{equation}
\label{E:expansion_return4}
\mE \mathds{1}_{\{T>K_0\}} \frac{|f^T_{ \omega}(x)-f^T_{ \omega}(y)|^\beta}{|x-y|^\beta}
\le 
\sum_{n=K_0+1}^\infty C_1\bigg(\frac{a}{x}\bigg)^\alpha(L')^{\beta n} e^{-\kappa n}< \tilde{\eta}.
\end{equation}

Combining \eqref{E:expansion_return3} with \eqref{E:expansion_return4} and using the fact that $\tilde{\eta}<1$ yields
$$\mE \frac{|f^T_{ \omega}(x)-f^T_{ \omega}(y)|^\beta}{|x-y|^\beta}
\le 
\bigg(1+
 \tilde{\eta}\bigg) 
 \bigg( \frac{C_1}{h} \frac{a^{1+\alpha}}{x^{1+\alpha}}  \bigg)^\beta .
$$
To finish the proof let us observe that $\frac{a}{x} = h^{-N}$ by the definition of $N$. Thus by $C_1^{N(1+\alpha)+1}\ge C_1$ we have
 $$\mE \frac{|f^T_{ \omega}(x)-f^T_{ \omega}(y)|^\beta}{|x-y|^\beta}
\le 
\bigg(1+
 \tilde{\eta}\bigg) 
 \bigg( \frac{C_1}{h^{N(1+\alpha)+1}}  \bigg)^\beta 
 \le 
 \bigg(1+
 \tilde{\eta}\bigg) 
 \bigg( \frac{C_1}{h} \bigg)^{\beta(N(1+\alpha)+1)}
 \le 
\bigg( 1+ \tilde{\eta}\bigg) 
 \bigg( \frac{C_1}{h} \bigg)^{3 \beta N}
 $$
It is easy to see that the proof is still valid if $T$ is replaced by $T\wedge n$.
\end{proof}

\section{Proofs}
\subsection{Proof of Theorem \ref{T:main}}
Let $f_1,\cdots, f_m$, $p_1,\cdots, p_m$ satisfy (A1)-(A3). Clearly it is sufficient to show that for every $a>0$ arbitrarily small
$$
\EXP\big| f^n_\omega(a) - f^n_\omega(1-a) \big| < C q^n
$$
for some $C\ge 1$ and $q\in (0,1)$.

Let $\Lambda$ be the volume Lyapunov exponent
$$
\Lambda= \sum_{i=1}^m p_i \int_0^1 \log f_i'(x) \mu_\ast(dx).
$$
By  Lemma 4.1 in \cite{Gharaei_Homburg_2017} $\Lambda<0$. Let $h$, $L'$, $L''$ be the same as in the preceding section. Let $\alpha$, $C_1$, $\kappa$ be given in Proposition \ref{P:basic1} and \ref{P:coupling} respectively, and let $a>0$ be such that Propositions \ref{P:basic1}, \ref{P:coupling}, \ref{P:expansion_return} are satisfied. Moreover, we assume
\begin{equation}
\label{E:proofT1}
3\mu((0,a))\log \frac{C_1}{h} < \Lambda/8, \quad 3\mu((1-a,1))\log \frac{C_1}{h} < \Lambda/8
\end{equation}
and
$$
hx \le f_i(x) \le h^{-1}x, \quad hx \le 1-f_i(x) \le h^{-1} x \quad \textrm{for $x\in (0,a)$.}
$$
Let $\varphi$ be any real function on $\Sigma\times [0,1]$ depending only on the first coordinate of $\omega$, continuous for $\omega\in\Sigma$ fixed and such that
$$
\varphi(\omega, x)\ge 3\log \frac{C_1}{h}  \mathds{1}_{(0,a)}(x)
+3\log \frac{C_1}{h}  \mathds{1}_{(1-a,1)}(x)
+\log f_\omega'(x) \mathds{1}_{(a,1-a)}(x).
$$
and (by \eqref{E:proofT1}) 
$$\int\int_{\Sigma\times[0,1]} \varphi \textrm{d}\mathbb{P}\otimes \mu_{\ast} <3/4\Lambda<0.$$
By Proposition \ref{P:uniform} there exist $k_0$ such that
\begin{equation}
\label{E:proofT2}
\EXP \varphi(x)+\cdots+\EXP \varphi(f^{k_0-1}_\omega(x))<\frac{k_0\Lambda}{2}<0 \quad \textrm{for $x\in [a,1-a]$.}
\end{equation}
If $r^-(\omega, x)$, $r^+(\omega, x)$ are the numbers of visits of $(f^n_\omega(x))$ in $(0,a)$ and $(1-a,1)$, respectively, until $n\le k_0$ then \eqref{E:proofT2} implies
$$
\EXP \log \bigg( (f^{k_0}_\omega)'(x) \bigg( \frac{C_1}{h} \bigg)^{3r^-(\omega, x)+3r^+(\omega, x)} \bigg) < \frac{k_0\Lambda}{2}<0
\quad \textrm{for $x\in [a,1-a]$,}
$$
and by applying the Taylor formula to $\beta\longmapsto c^\beta$ for $\beta$ in the neighborhood of 0, $c$-fixed number, we obtain $\eta>0$ and $\beta\in (0,1)$ with
$$
\EXP\bigg( (f^{k_0}_\omega)'(x) \bigg( \frac{C_1}{h} \bigg)^{3r^-(\omega, x)+3r^+(\omega, x)} \bigg)^\beta < 1-\eta.
$$
We assume also $\beta$ to be so small that Proposition \ref{P:expansion_return} is satisfied.  By the mean value theorem and the continuity for every $x, y\in[a,1-a]$ with $x<y$ and $|x-y|$ sufficiently small (say less than $\delta$)
\begin{equation}
\label{E:proofT3}
\EXP |f^{k_0}_\omega(x)- f^{k_0}_\omega(y) |^\beta \bigg(  \frac{C_1}{h} \bigg)^{3\beta r^-(\omega, x)+3\beta r^+(\omega, y)}<1-\eta.
\end{equation}
Let as assume also that $\delta$ is so small that the points $|f^{k_0}(x)-f^{k_0}(y)|<\delta'$ for $\omega\in \Sigma$, where $\delta'$ is given in Proposition \ref{P:expansion_return}.

We now proceed to define now a stopping time $\tau=\tau_{x,y}(\omega)$.
\vspace{0.5cm}

\noindent \textit{Case $|x-y|<\delta$}

 If $|x-y|<\delta$ then put
$$
\tau_{x,y}(\omega) = \min \{k\ge k_0: f^k_\omega(x), f^k_\omega(y)\in [a,1-a] \}.
$$
By the definition of $h$, if $f^{k_0}_\omega(x)<a$, then $f^{k_0}_\omega(x)\ge h^{r^-(\omega,x)} a$, thus $N(\omega):=-\frac{\log (a/f^{k_0}_\omega(x))}{\log h}<r^-(\omega)<r^-(\omega)+r^+(\omega)$ for every $\omega\in \Sigma$. For the same reason $N(\omega)<r^-(\omega)+r^+(\omega)$ if $f^{k_0}_\omega(y)>1-a$. By this fact, Proposition \ref{P:expansion_return} applied to $\tilde{\eta}=\eta$, the strong Markov property and \eqref{E:proofT3}
$$
\EXP|f^{\tau\wedge n}_\omega(x) - f^{\tau\wedge n}_\omega(y)|^\beta
\le
(1+\eta)\EXP \bigg(\frac{C_1}{h}\bigg)^{3\beta N(\omega)}|f^{k_0}_\omega(x) - f^{k_0}_\omega(y)|^\beta
$$
$$
\le 
(1+\eta)\EXP \bigg(\frac{C_1}{h}\bigg)^{3\beta (r^-(\omega)+r^+(\omega))}|f^{k_0}_\omega(x) - f^{k_0}_\omega(y)|^\beta
< (1+\eta)(1-\eta)|x-y|^\beta=(1-\eta^2)|x-y|^\beta,
$$
thus
\begin{equation}
\label{E:proofT4}
\EXP|f^{\tau\wedge n}_\omega(x) - f^{\tau\wedge n}_\omega(y)|^\beta <
(1-\eta^2)|x-y|^\beta \quad \textrm{for $|x-y|<\delta$.}
\end{equation}
\vspace{0.5cm}

\noindent \textit{Case $|x-y|\ge \delta$}

By the synchronization of the system there exists $k_1$ such that for
$$
\Prob\bigg( \exists_{k\ge k_1} |f^k_\omega(a)-f^k_\omega(1-a)|^\beta \ge \frac 1 2(1-\eta^2)\delta^\beta \bigg) 
\le \frac{1}{2}(1-\eta^2)\delta^\beta.
$$ 
If for $x,y$ with $|x-y|\ge \delta$ we define $\tau=\tau_{x,y}(\omega)$ to be the first $k\ge k_1$ with $f^k_\omega(x), f^k_\omega(y)\in [a,1-a]$, then the above easily implies
\begin{equation}
\label{E:proofT5}
\EXP|f^{\tau\wedge n}_\omega(x) - f^{\tau\wedge n}_\omega(y)|^\beta <
(1-\eta^2)|x-y|^\beta \quad \textrm{for $|x-y|\ge \delta$, $n\ge k_1$.}
\end{equation}
\vspace{0.5cm}

In both cases $\tau$ is defined as the first moment greater or equal to either $k_0$ or $k_1$ that the trajectory visits $[a,1-a]$. Therefore Proposition \ref{P:coupling} shows that $\sup_{x,y\in [a,1-a]} \EXP \exp(\kappa \tau_{x,y})<\infty$. Let
$$
\rho_n (\omega) = \max \{ j\ge 1: \tau_1+\cdots+\tau_j\le n \},
$$
where $\tau_{j+1} = \tau_j(\omega)+\tau_{x', y'}(\theta^{\tau_j} \omega)$, $x'=f^{\tau_j}_\omega(x)$, $y'=f^{\tau_j}_\omega(y)$, $j\ge 1$, $\tau_1=\tau$. By a similar argument like in \eqref{E:large_deviation6} there exists $\lambda\in (0,1)$ such that $\Prob(\rho_n < \lambda n)$ decays exponentially fast uniformly over $x,y\in[a,1-a]$.

Let 
$$
L=\max_{\omega \in \Sigma}
	 \max\{ \|(f^{i}_\omega)' \|_\infty : i=1,2, \ldots, \max\{k_0, k_1\} \}.
	 $$ 
Fix $n$ and $l\ge \lambda n$ and define $B = \{ \omega\in\Sigma : \rho_n=l \}$. 
Combining \eqref{E:proofT4}, \eqref{E:proofT5}, the strong Markov property and the definition of $L$ we get
$$
\EXP \mathds{1}_{B_l}|f^{ n}_\omega(x) - f^{n}_\omega(y)|^\beta
=
\EXP \EXP[ \mathds{1}_{B_l}|f^{ n}_\omega(x) - f^{n}_\omega(y)|^\beta | \mathcal{F}_{\tau_l} ]
\le 
L \EXP \mathds{1}_{B_l} |f^{\tau_l}_\omega(x) - f^{\tau_l}_\omega(y)|^\beta
$$
$$
\le 
L\EXP\mathds{1}_{B_l} |f^{\tau_{\lfloor \lambda n\rfloor}}_\omega(x) - f^{\tau_{\lfloor \lambda n\rfloor}}_\omega(y)|^\beta,
$$
therefore summation over all $l\ge \lambda n$ leads to
$$
\EXP |f^{ n}_\omega(x) - f^{n}_\omega(y)|^\beta
\le 
L \EXP |f^{\tau_{\lfloor \lambda n\rfloor}}_\omega(x) - f^{\tau_{\lfloor \lambda n\rfloor}}_\omega(y)|^\beta 
+
\Prob(\rho_n < \lambda n).
$$
Applying again the strong Markov property, \eqref{E:proofT4}, \eqref{E:proofT5} and the fact that $c\le c^\beta$ for $c\in [0,1]$ we get
$$
\EXP|f^{ n}_\omega(x) - f^{n}_\omega(y)|\le  \EXP |f^{ n}_\omega(x) - f^{n}_\omega(y)|^\beta
\le 
L (1-\eta^2)^{ \lfloor \lambda n \rfloor}|x-y|^{\beta}
+\Prob(\rho_n<\lambda n),
$$
which completes the proof.

\subsection{Proof of Corollary \ref{C:2}}
In the first step we are going to show that there exist $\tilde{C}\ge 1$, $\widetilde{q}\in (0,1)$ such that
\begin{equation}\label{E:C2}
\bigg|U^n \varphi(x) - \int_0^1 \varphi \textrm{d}\mu_\ast\bigg| \le \tilde{C}L\bigg(\frac{a}{\min\{x,1-x\}}\bigg)^\alpha \widetilde{q}^{\nu n} \quad \textrm{for $n\ge 1$.}
\end{equation}
for every H{\"o}lder function $\varphi$ with $|\varphi(x)-\varphi(y)|\le L|x-y|^\nu$.

Let $a>0$ be such that $\EXP|f^n_\omega(a)-f^n_\omega(1-a)|\le Cq^n$, $n\ge 1$. Let $\varphi$ be as above.  We have by Theorem \ref{T:main}
\begin{equation}
\label{E:proofT6}
|U^n\varphi(y)-U^n\varphi(y') |
=
|\EXP\varphi(f^n_\omega(y))-\EXP\varphi(f^n_\omega(y'))|
\le 
L \EXP \big| f^n_\omega(y) - f^n_\omega(y') \big|^\nu \le
L C^\nu q^{\nu n}
\end{equation}
for $y,y'\in [a,1-a]$.  First we are going to show the assertion under the assumption that $x\in [a,1-a]$. We have
  $$
  \bigg|U^n\varphi(x)-\int_{[0,1]}\varphi(y)\mu_\ast(dy)\bigg|
  =
  \bigg|U^n\varphi(x)-\int_{[0,1]}\varphi(y)P^n\mu_\ast(dy)\bigg|
  $$
  $$
 = \bigg|U^n\varphi(x)-\int_{[0,1]}U^n\varphi(y)\mu_\ast(dy)\bigg|
  \le 
 L \int_{[0,1]}\EXP \big| f^n_\omega(x) - f^n_\omega(y) \big|^\nu \mu_\ast(dy)
 $$
 
Let $\gamma>0$ be such that $\alpha\gamma - 1/2<0$, where $\alpha$ is given in Proposition \ref{P:basic1}. Let $\xi_n = ac^{-\gamma n}$. Then $\mu_\ast((0,\xi_n)\cup (1-\xi_n, 1))\le 2M\xi_n^\alpha$, as $\mu_\ast\in\mathcal{P}_{M,\alpha}$  by Proposition \ref{P:basic1}. Let $R_n^-= \bigcap_{j=1}^{n/2} \{\omega :  f^j_\omega(\xi_n)<a \} $, $R_n^+= \bigcap_{j=1}^{n/2} \{ \omega :  f^j_\omega(\xi_n)>1-a \} $. Using Proposition \ref{P:basic1} again we get that $\Prob(R_n^-\cup R_n^+)\le c^{n(1/2-\alpha\gamma)}$. Combining these two facts we reduce the proof to showing that
$$
\int_{\xi_n}^{1-\xi_n} \int_{\Sigma\setminus (R_n^-\cup R_n^+)} 
\big| f^n_\omega(x) - f^n_\omega(y) \big|^\nu \mu_\ast(dy)
$$
decays exponentially fast. Let, having $y\in [\xi_n, 1-\xi_n]$ fixed,
$$C_l = \{ f^l_\omega(y)\in [a,1-a] \}
\cap \{f_\omega^{l-1}(y)\not\in [a,1-a] \} 
\cap \cdots 
\cap \{f_\omega^{1}(y)\not\in [a,1-a] \}
\cap \{ y\not\in [a,1-a] \}
$$
for $l=0,1\cdots, \lfloor n/2 \rfloor$. By the Markov property and \eqref{E:proofT6}
$$
\int_{C_l} \big| f^n_\omega(x) - f^n_\omega(y) \big|^\nu \textrm{d}\mathbb{P} < C^\nu q^{\nu n/2}\Prob(C_l).
$$
Since $\Sigma\setminus (R_n^-\cup R_n^+)$ is a disjoint sum of $C_l$'s and the above estimate works for every $y\in [\xi_n, 1-\xi_n]$,
\begin{equation}
\label{E:proofT7}
 \bigg|U^n\varphi(x)-\int_{[0,1]}\varphi(y)\mu_\ast(dy)\bigg|
  \le
  c^{n(1/2-\alpha\gamma)} +  C^\nu q^{\nu n/2} + 2M\xi_n^\alpha
\end{equation}
provided $x\in [a,1-a]$.

For $x\not\in [a,1-a]$ it is convenient to assume without loss of generality that $\int_0^1 \varphi d\mu_\ast=0$. 
 Let
$$C_l = \{ f^l_\omega(x)\in [a,1-a] \}
\cap \{f_\omega^{l-1}(x)\not\in [a,1-a] \} 
\cap \cdots 
\cap \{f_\omega^{1}(x)\not\in [a,1-a] \}
\cap \{ x\not\in [a,1-a] \}
$$
for $l=1,2\cdots, \lfloor n/2 \rfloor$, and $R_n= \bigcap_{j=1}^{\lfloor n/2 \rfloor} \{ f^j_\omega(x)<a \} $  we have
$$
\bigg|U^n\varphi(x)\bigg|
\le 
\sum_{l=1}^{\lfloor n/2 \rfloor} \int_{C_l}\bigg| U^{n-l}\varphi(f^l_\omega(x)) \bigg|
+ \Prob(R_n) \|\varphi\|_\infty
  $$
  
By \eqref{E:proofT7} the first term decays not slower than $c^{\frac n2(1/2-\alpha\gamma)} +  C^\nu q^{\nu n/4} + 2M\xi_{\lfloor n/2 \rfloor}^\alpha$ (note \eqref{E:proofT7} is uniform over $x\in[a,1-a]$). By Proposition \ref{P:basic1} $\Prob(R_n)\le (\frac{a}{z} )^\alpha c^{n/2}$, where $z=\min\{x, 1-x\}$. This proves \eqref{E:C2}.

To pass to the assertion in the statement, fix two Lipschitz functions $\psi$, $\varphi$. Then \eqref{E:C2} holds for $\varphi$ with $\nu=1$. Let $\zeta$ be such that $\zeta^{\alpha} \in (\widetilde{q}, 1)$. Let $K_n = [a\zeta^n, 1- a\zeta^n]\subseteq [0,1]$ and $R_n = [0,1]\setminus K_n$. We have by \eqref{E:C2} and Proposition \ref{P:basic1} again

$$\bigg|\int_0^1  U^n \varphi(x) \psi(x) d\mu_\ast(x) - \int_0^1 \varphi(x) \textrm{d}\mu_\ast(x) \int_0^1 \psi(x) \textrm{d}\mu_\ast(x)  \bigg| 
\le 
\int_0^1 \bigg| U^n \varphi(x) - \int_0^1 \varphi(x) \textrm{d}\mu_\ast(x) \bigg| \|\psi\|_\infty d\mu_\ast(x)
$$
$$
=\int_{K_n} \bigg| U^n \varphi(x) - \int_0^1 \varphi(x) \textrm{d}\mu_\ast(x) \bigg| \|\psi\|_\infty d\mu_\ast(x)
+
\int_{R_n} \bigg| U^n \varphi(x) - \int_0^1 \varphi(x) \textrm{d}\mu_\ast(x) \bigg| \|\psi\|_\infty d\mu_\ast(x),
$$
$$
\le \widetilde{C} \|\varphi\|_{\Lip} \|\psi\|_{\Lip} \bigg( \frac{a}{a\zeta^n} \bigg)^\alpha \widetilde{q}^n
 + 4 M a^\alpha \zeta^{\alpha n} \|\varphi\|_{\Lip} \|\psi\|_{\Lip},
$$
which immediately implies the assertion by the choice of $\zeta$.

\bibliographystyle{alpha}
\bibliography{bib}

\end{document}